\documentclass{amsart}

\usepackage{amsmath,amssymb,amsthm}
\usepackage{graphicx}
\usepackage{fullpage}
\usepackage{xcolor,colortbl}
\usepackage{setspace}
\usepackage{paralist}

\newtheoremstyle{myremark} 
    {7pt}                    	
    {7pt}                    	
    {}  	                 	
    {}                        	
    {\bf}       	      	
    {.}                          
    {.5em}                  	
    {}  				

\theoremstyle{plain}
\newtheorem{lemma}{Lemma}[section]
\newtheorem{theorem}[lemma]{Theorem}
\newtheorem{definition}[lemma]{Definition}
\newtheorem{corollary}[lemma]{Corollary}
\newtheorem{conjecture}[lemma]{Conjecture}
\newtheorem{proposition}[lemma]{Proposition}

\newtheorem*{theorem*}{Main result}

\theoremstyle{myremark}
\newtheorem{remark}[lemma]{Remark}
\newtheorem{example}[lemma]{Example}

\newtheorem*{application*}{Applications}

\newcommand{\R}{\mathbb{R}}
\newcommand{\Z}{\mathbb{Z}}
\newcommand{\nB}{\mathcal{B}}
\newcommand{\nN}{\mathcal{N}}
\newcommand{\onN}{\overline{\mathcal{N}}}
\newcommand{\nU}{\mathcal{U}}
\newcommand{\cnk}[2]{\mathcal{N}({#1},{#2})}
\newcommand{\vnk}[2]{\overline{\mathcal{N}}({#1},{#2})}
\newcommand{\cech}{\mathrm{\breve{C}ech}}
\newcommand{\cl}{\mathrm{Cl}}
\newcommand{\conn}{\mathrm{conn}}
\newcommand{\conv}{\mathrm{conv}\,}
\newcommand{\infimum}{\mathrm{inf}}
\newcommand{\lk}{\mathrm{lk}}
\newcommand{\md}{\mathrm{~mod~}}
\newcommand{\redhom}{\widetilde{H}}

\newcommand{\tbeta}{\widetilde{\beta}}
\newcommand{\tgamma}{\widetilde{\gamma}}
\newcommand{\verteq}{\rotatebox{90}{$\,=$}}
\newcommand{\vr}{\mathrm{VR}}

\renewcommand{\subset}{\subseteq}

\begin{document}
\title{Nerve complexes of circular arcs}
\author{Micha{\l} Adamaszek}
\address{Max Planck Institut f\"{u}r Informatik, 66123 Saarbr\"{u}cken, Germany}
\email{aszek@mimuw.edu.pl}
\author{Henry Adams}
\address{Department of Mathematics, Duke University, Durham, NC 27708, United States}
\email{hadams@math.duke.edu}
\author{Florian Frick}
\address{Institut f\"{u}r Mathematik, MA 8-1, Technische Universit\"{a}t Berlin, 10623 Berlin,
Germany}
\email{frick@math.tu-berlin.de}
\author{Chris Peterson}
\address{Department of Mathematics, Colorado State University, Fort Collins, CO 80523, United States}
\email{peterson@math.colostate.edu}
\author{Corrine Previte--Johnson}
\address{Department of Mathematics, California State University, San Bernardino, CA 92407, United States}
\email{cprevite@csusb.edu}
\thanks{Research of HA was supported by the Institute for Mathematics and its Applications. FF is supported by the German Science Foundation DFG via the Berlin Mathematical
School}

\begin{abstract}
\begin{spacing}{1.1}
{\normalsize
We show that the nerve complex of $n$ arcs in the circle is homotopy equivalent to either a point, an odd-dimensional sphere, or a wedge sum of spheres of the same even dimension. Moreover this homotopy type can be computed in time $O(n\log n)$. For the particular case of the nerve complex of evenly-spaced arcs of the same length, we determine the dihedral group action on homology, and we relate the complex to a cyclic polytope with $n$ vertices. We give three applications of our knowledge of the homotopy types of nerve complexes of circular arcs. First, we use the connection to cyclic polytopes to give a novel topological proof of a known upper bound on the distance between successive roots of a homogeneous trigonometric polynomial. Second, we show that the Lov\'{a}sz bound on the chromatic number of a circular complete graph is either sharp or off by one. Third, we show that the Vietoris--Rips simplicial complex of $n$ points in the circle is homotopy equivalent to either a point, an odd-dimensional sphere, or a wedge sum of spheres of the same even dimension, and furthermore this homotopy type can be computed in time $O(n\log n)$.}
\end{spacing}
\end{abstract}

\maketitle
\setcounter{tocdepth}{1}
\tableofcontents

\section{Introduction}

For $\nU$ a collection of subsets of some topological space, the \emph{nerve simplicial complex} $\nN(\nU)$ contains a $k$-simplex for every subcollection of $k+1$ sets with nonempty intersection. The Nerve Theorem, which holds in a variety of contexts, states that if the intersection of each subcollection of $\nU$ is either empty or contractible, then the nerve complex is homotopy equivalent to the union of the subsets \cite{Borsuk1948, Bjorner1995}. A coarser representation of the incidences between sets in $\nU$ is given by the \emph{clique complex} $\onN(\nU)$, which contains a $k$-simplex for every collection of $k+1$ sets with \emph{pairwise} nonempty intersections. In this paper we study nerve complexes and clique complexes of finite collections of arcs in the circle, which are known, respectively, as \emph{ambient \u Cech complexes} and \emph{Vietoris-Rips complexes} when all arcs have the same length. We completely classify their homotopy types.

\begin{theorem*}[Theorem \ref{thm:arbitraryNerve}]
The nerve complex and the clique complex of any finite collection of arcs in the circle are homotopy equivalent to either a point, an odd-dimensional sphere, or a wedge sum of spheres of the same even dimension. 
\end{theorem*}

The higher-dimensional spheres occur when the arcs are large enough so that the intersection of two arcs need not be contractible.

We begin by studying the homotopy types and the combinatorics of the nerve complexes of evenly-spaced circular arcs. For $0 \le k < n$, let $\cnk{n}{k}$ denote the nerve complex of $n$ evenly-spaced arcs each occupying a $\frac{k}{n}$ fraction of the circumference of the circle. These are of fundamental interest since, as we shall see, the nerve of any finite configuration of arcs deformation retracts to a complex isomorphic to some $\cnk{n}{k}$. We prove a recursive relation from which we derive the homotopy types of the $\cnk{n}{k}$. We further provide explicit generators of homology and cohomology of $\cnk{n}{k}$ and describe the induced action of their automorphism groups on homology. In the generic case, when $\cnk{n}{k}$ is homotopy equivalent to an odd-dimensional sphere $S^{2l+1}$, we show that it contains the boundary complex of the $n$-vertex cyclic polytope $C_{2l+2}(n)$ as a homotopy equivalent subcomplex. 

\begin{application*}
We give two immediate applications of these calculations. 
\begin{compactenum}[1.]
	\item We show that the Lov\'{a}sz bound on the chromatic number of a circular complete graph is either sharp or off by one (see Corollary \ref{cor:lovasz}).
	\item We use the relation with cyclic polytopes to give a novel topological proof of a known upper bound on the distance between successive roots of a homogeneous trigonometric polynomial (see Theorem \ref{thm:gapsAlternate}).
\end{compactenum}
\end{application*}

In the last section we study arbitrary collections $\nU$ of $n$ arcs in $S^1$. We show that each such nerve complex $\nN(\nU)$ has an explicit homotopy-preserving combinatorial reduction to one of the form $\cnk{n'}{k}$ with $n' \le n$. We can compute the reduction in time $O(n \log n)$, and as a result we obtain an efficient algorithm for determining the homotopy type of $\nN(\nU)$, even though the worst-case size of $\nN(\nU)$ is exponential in $n$. The reductions are independent of any knowledge of the $\cnk{n}{k}$, and they also carry over to the clique complexes $\onN(\nU)$.
 
When $\nU$ is a collection of balls of fixed radius $r$ in a Riemannian manifold $M$, the clique (or flag) complex $\onN(\nU)$ is called a \emph{Vietoris--Rips complex} \cite{Vietoris27}. Such complexes arise in manifold reconstruction \cite{ChazalOudot2008, AttaliLieutierSalinas2013} and in topological data analysis \cite{EdelsbrunnerHarer, Carlsson2009}. If the radius $r$ is sufficiently small and if the balls are sufficiently dense, then Hausmann and Latschev prove the Vietoris--Rips complex is homotopy equivalent to manifold $M$ \cite{Hausmann1995, Latschev2001}. However, Vietoris--Rips complexes with larger radii parameters $r$ are not well understood, even for simple spaces such as spheres. We show that the Vietoris--Rips complex of an arbitrary subset of $n$ points in the circle with arbitrary radius parameter $r$ is homotopy equivalent to either a point, an odd-dimensional sphere, or a wedge sum of spheres of the same even dimension. Moreover this homotopy type can be computed in time $O(n\log n)$. 
We also prove a surprising relationship, different from the usual inclusion, between the \u Cech and Vietoris--Rips complexes of evenly-spaced points on the circle.

\section{Preliminaries}

We assume the reader is familiar with basic concepts in topology and combinatorial topology, and refer to Hatcher \cite{Hatcher} and Kozlov \cite{Kozlov}.

\subsection*{Simplicial complexes} Let $K$ be a simplicial complex, let $V(K)$ be its vertex set, and let $K^{(i)}$ be its $i$-skeleton. We will identify an abstract complex with its geometric realization and use the symbol $\simeq$ to denote homotopy equivalence and $\cong$ to denote isomorphism of simplicial complexes. For $V' \subseteq V(K)$, let $K[V']$ be the induced subcomplex of $K$ containing only those simplices with all vertices in $V'$. We let $K \setminus \{v\} = K[V(K) \setminus \{v\}]$ be the simplicial complex obtained from $K$ by removing all simplices containing $v$.
The link of vertex $v$ is $\lk_K(v) = \{\sigma \in K~|~v\notin \sigma\mbox{ and }\sigma \cup \{v\} \in K\}$. 

\subsection*{Domination}
We say vertex $v$ is \emph{dominated} by vertex $v'$ if each $\sigma \in K$ containing $v$ satisfies $\sigma \cup \{v'\} \in K$, i.e.\ if $\lk_K(v)$ is a cone with apex $v'$. If vertex $v \in K$ is dominated, then $K \simeq K \setminus \{v\}$ because we are removing a vertex $v$ whose link is contractible. In fact there is a deformation retraction $K\mapsto K\setminus\{v\}$ which sends $v$ to $v'$ and also  simplicially collapses $K$ to $K\setminus\{v\}$. These removals go by various names: folds, elementary strong collapses, and LC reductions \cite{BabsonKozlov2006, BarmakMinian2012, Matouvsek2008}. An analogous operation for graphs is known as dismantling. 

We say that simplicial complex $K$ is \emph{minimal} if it contains no dominated vertices.

\subsection*{Nerves and cliques}
Let $Y$ be a topological space, and let $\nU = \{U_i\}_{i\in I}$ with $\emptyset \neq U_i \subseteq Y$ be a collection of subsets. We say $\nU$ is a \emph{covering} of $Y$ if $\bigcup_{i\in I}U_i = Y$.

\begin{definition}
Given a collection of subsets $\nU=\{U_i\}_{i\in I}$ in a topological space, the \emph{nerve simplicial complex} $\nN(\nU)$ has vertex set $I$ and contains $k$-simplex $[i_0, \dots, i_k]$ if $\bigcap_{j=0}^k U_{i_j} \neq \emptyset$.
\end{definition}
The Nerve Theorem is generally attributed to Borsuk \cite{Borsuk1948}, and the version we use is due to Bj{\"o}rner {\cite[Theorem~10.6]{Bjorner1995}.

\begin{theorem}[Nerve Theorem]\label{thm:nerve}
Let $K$ be a simplicial complex and let $\nU$ be a covering by subcomplexes. If every nonempty finite intersection of complexes in $\nU$ is contractible, then $\nN(\nU) \simeq K$.
\end{theorem}

In Sections~\ref{sec:clique_even} and \ref{sec:arbitrary} we will consider clique complexes of nerves. For $G$ a simple, loopless, undirected graph, the \emph{clique simplicial complex} $\cl(G)$ has $V(G)$ as its vertex set and a face for each clique (complete subgraph) of $G$.

\begin{definition}
Given a collection of subsets $\nU=\{U_i\}_{i\in I}$ in a topological space, the \emph{clique simplicial complex} $\onN(\nU)$ has vertex set $I$ and contains $k$-simplex $[i_0, \dots, i_k]$ if $U_{i_j} \cap U_{i_{j'}} \neq \emptyset$ for all $0 \le j,j' \le k$.
\end{definition}
We note $\onN(\nU) = \cl(\nN(\nU)^{(1)})$.

\subsection*{\u Cech and Vietoris--Rips complexes}
In the particular case when $Y$ is a metric space and $\nU$ is a collection of balls, the nerve complex is also known as a \u Cech complex, and the clique complex is also known as a Vietoris--Rips complex. For $Y$ a metric space, we denote the closed ball of radius $r \geq 0$ centered at $y \in Y$ by $B(y,r) = \{y' \in Y~|~d(y,y')\leq r\}$. Fix some $X \subseteq Y$ and let $\nU(X,r) = \{B(x,r)~|~x \in X\}$. Then $\nN(\nU(X,r))$ is isomorphic to the \emph{ambient \u Cech complex} $\cech(X,Y;r)$ with landmark set $X$ and witness set $Y$, as defined by Chazal, de~Silva, \& Oudot \cite[Section~4.2.3]{ChazalDeSilvaOudot2013}\footnote{Attali \& Lieutier \cite{AttaliLieutier2014} refer to the ambient \u Cech complex as a restricted \u Cech complex.}. The \emph{Vietoris--Rips complex} $\vr(X,r)$ is defined to be the simplicial complex on vertex set $X$ containing finite $\sigma \subset X$ as a simplex if the distance between any two points in $\sigma$ is at most $r$. If $Y$ is a geodesic space then $\vr(X,r)$ is isomorphic to $\onN(\nU(X,r/2))$.

\subsection*{Conventions regarding $S^1$}
In this paper we study the setting where $\nU$ is a finite collection of arcs in the circle $S^1$. We identify $S^1$ with $\R/\Z$, where the positive orientation on $\R$ corresponds to the clockwise orientation on $S^1$. For $x,y \in \R$ with $x \le y$ we denote by $[x,y]_{S^1}$ the \emph{closed circular arc} obtained as the image of the interval $[x,y]$ under the quotient map $\R \to \R/\Z$. Similarly, for $a,b \in S^1$ we denote by $[a,b]_{S^1}$ the closed circular arc obtained by moving from $a$ to $b$ in a clockwise fashion. Open and half-open intervals in $S^1$ are obtained by removing endpoints from closed intervals. The intersection of $k$ such arcs is either empty, contractible, or homotopy equivalent to a disjoint union of at most $k$ points; $\nU$ is known as an acyclic family \cite{ColinDeVerdiereGinotGoaoc2012}. 

We also equip the circle $S^1$ of circumference $1$ with the natural arc-length distance. Under this metric the diameter of $S^1$ is $\frac{1}{2}$. The choice of this particular metric does not influence the generality of our results.

\subsection*{Other conventions}
We denote the topological space consisting of a single point by $\ast$. For a topological space $Y$ we let $\bigvee^i Y$ denote the wedge sum of $i$ copies of $Y$, where by convention $\bigvee^0 Y = \ast$. The symbol $\Sigma$ denotes unreduced suspension.

All homology and cohomology is taken with integer coefficients.

If $\sigma$ is an oriented $d$-simplex in $K$ (an element of the standard basis of the chain group $C_d(K)$) then $\sigma^\vee$ denotes the dual $d$-cochain which assigns $1$ to $\sigma$, $-1$ to the reverse oriented $\sigma$, and $0$ to other $d$-simplices.

\section{Nerve complexes of evenly-spaced arcs}

We begin by giving a combinatorial model for nerve complexes of evenly-spaced circular arcs.

\begin{definition}
For $n \geq 1$ and $i,j \in \Z$ with $i \le j$, let the \emph{discrete circular arc} $[i, j]_n$ be the image
of the set $\{i, i+1, \dots, j\}$ under the quotient map $\Z \to \Z/n, z \mapsto z \md n$.
\end{definition}

For most of this paper we will be studying the topology and combinatorics of the following family of abstract simplicial complexes.

\begin{definition}
For $n \ge 1$ and $k \ge 0$, the \emph{nerve complex} $\cnk{n}{k}$ has vertex set $\{0, \dots, n-1\}$, and its set of maximal simplices is $\bigl\{[i, i+k]_n~|~i=0, \dots, n-1\bigr\}$.
\end{definition}
If $k \le n-2$ then $\cnk{n}{k}$ has $n$ maximal simplices given by the $n$ rotations of $[0,k]_n$, and if $k \ge n-1$ then $\cnk{n}{k}$ is the $(n-1)$-simplex.

To see the connection with evenly-spaced circular arcs, for $0 \le k < n$ consider the collection
\begin{equation}\label{eq:evenly-spaced}
\nU_{n,k} = \Bigl\{\Bigl[\frac{i}{n}, \frac{i+k}{n}\Bigr]_{S^1}~\Big|~i=0, \dots, n-1\Bigr\}
\end{equation}
of $n$ evenly-spaced arcs of length $\frac{k}{n}$. Also, let $X_n \subset S^1$ be a set of $n$ evenly-spaced points. It is an easy exercise to verify the isomorphisms of simplicial complexes
\begin{equation}
\label{eq:cech-iso}
\cech\bigl(X_n,S^1;\tfrac{k}{2n}\bigr)\cong\nN(\nU_{n,k})\cong\cnk{n}{k}.
\end{equation}
For $k$ even, the complex $\cnk{n}{k}$ can also be described as a distance-neighborhood complex of the cycle graph $C_n$, as studied by the last author \cite{Previte2014}.

The following regimes are simple.
\begin{itemize}
\item Disconnected: $\cnk{n}{0}$ is the disjoint union of $n$ points, i.e.\ $\bigvee^{n-1} S^0$.
\item Circle: For $1 \le k < n/2$ we have $\cnk{n}{k} \simeq S^1$ by the Nerve Theorem. Indeed, consider the triangulation of $S^1$ with vertices $\frac{i}{n}$ and edges $[\frac{i}{n},\frac{i+1}{n}]_{S^1}$ for $i = 0, \dots, n-1$. Since $1 \le k < n/2$, the covering $\nU_{n,k}$ of $S^1$ has all nonempty intersections contractible. We have $\cnk{n}{k} \cong \nN(\nU_{n,k})$, and Theorem~\ref{thm:nerve} gives $\nN(\nU_{n,k}) \simeq S^1$.
\item Top-dimensional sphere: $\cnk{n}{n-2}$ is the boundary of the $(n-1)$-simplex.
\item Contractible: For $k \geq n-1$ the complex $\cnk{n}{k}$ is the full $(n-1)$-simplex.
\end{itemize}

\begin{example}\label{ex:cnk63}
The nerve complex $\cnk{6}{3}$ is the nerve of the $6$ equally-spaced closed arcs of length $\frac{3}{6}=\frac12$; see Figure~\ref{fig:cnk63}. The Nerve Theorem does not apply since $[\frac{i}{6},\frac{i+3}{6}]_{S^1} \cap [\frac{i+3}{6},\frac{i}{6}]_{S^1} \simeq S^0$ is not contractible.
The complex $\cnk{6}{3}$ has six maximal $3$-simplices, and as we shall see $\cnk{6}{3} \simeq \bigvee^2 S^2$.
\begin{figure}[h]
	\begin{center}
    	\includegraphics[width=4.5in]{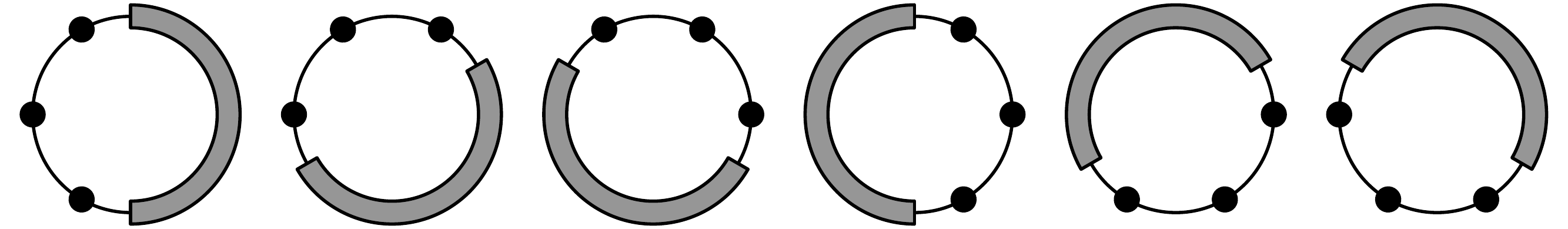}
	\includegraphics[width=4.5in]{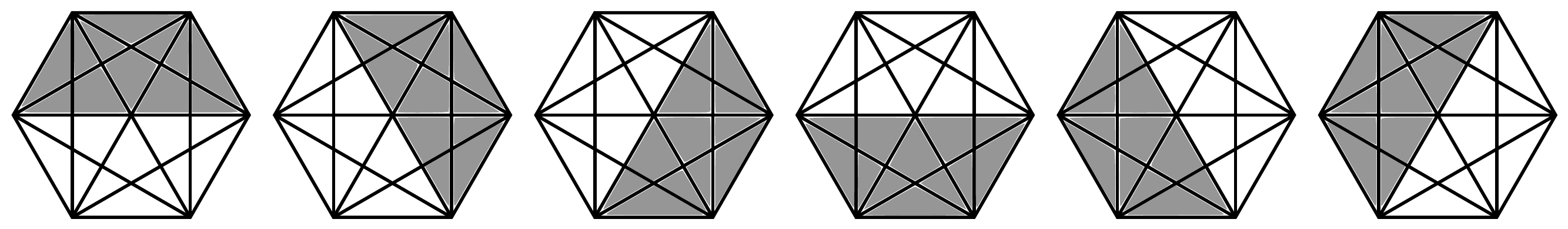}
	\end{center}
	\caption{(Top) The six arcs of $\nU_{6,3}$. (Bottom) The six maximal 3-simplices in $\cnk{6}{3}$}
	\label{fig:cnk63}
\end{figure}
\end{example}

We will now determine the homotopy types of the complexes $\cnk{n}{k}$. For this we repeatedly use the following lemma, which is a simple version of \cite[Lemma 10.4.(ii)]{Bjorner1995}.

\begin{lemma}\label{lem:union}
If the simplicial complex $K$ is the union of two contractible subcomplexes $K_1$ and $K_2$, then $K\simeq \Sigma\,(K_1\cap K_2)$.
\end{lemma}

\begin{proposition}\label{prop:cnk_susp}
For $n/2\le k < n$ we have $\cnk{n}{k}\simeq \Sigma^2\,\cnk{k}{2k-n}$.
\end{proposition}

\begin{proof}
Denote the maximal simplices of $\cnk{n}{k}$ by $\sigma_i = [i,i+k]_n$ for $i = 0, \dots, n-1$. Then we can write
\[ \cnk{n}{k} = \Big(\bigcup_{i=0}^{n-k-2}\sigma_i\Big) \cup \Big(\bigcup_{j=n-k-1}^{n-1} \sigma_j\Big), \]
where by a slight abuse of notation we write $\bigcup_{t\in T}\sigma_t$ for the subcomplex of $\cnk{n}{k}$ with maximal simplices $\{ \sigma_t~|~ t\in T \}$. Each $\sigma_j$ contains $n-1$ and each $\sigma_i$ contains $k$ since $n-k-2\leq k$, hence both unions are cones. Moreover, the simplices $\sigma_i$ do not contain $n-1$. By Lemma~\ref{lem:union} we have $\cnk{n}{k}\simeq \Sigma\,K$, where $K$ is the complex with vertex set $\{0,\dots,n-2\}$ whose maximal simplices are the inclusion-wise maximal elements in the family
\[ \{\sigma_i\cap \sigma_j~|~i = 0, \dots, n-k-2 \mbox{ and } j = n-k-1, \dots, n-1\}. \]
The intersections $\sigma_i\cap \sigma_j$ fall into three categories, see Figure~\ref{fig:maximalFaces}.
\begin{itemize}
\item[a)] If $0\leq i\leq j+k-n\leq i+k<j\leq n-1$ then $\sigma_i\cap \sigma_j=\{i,\ldots,j+k-n\}$. We have $\sigma_i\cap\sigma_j\subset\sigma_0\cap\sigma_{n-1}=\{0,\ldots,k-1\}$.
\item[b)] If $0\leq j+k-n<i\leq j\leq i+k\leq n-2$  then $\sigma_i\cap \sigma_j=\{j,\ldots,i+k\}$. We have $\sigma_i\cap\sigma_j\subset\sigma_{n-k-2}\cap\sigma_{n-k-1}=\{n-k-1,\ldots,n-2\}$.
\item[c)] If $i\leq j+k-n$ and $j\leq i+k$ then $\sigma_i\cap\sigma_j=\{i,\ldots,j+k-n\}\cup\{j,\ldots,i+k\}$. These are not contained in any other set of the form $\sigma_{i'}\cap \sigma_{j'}$.
\end{itemize}
We conclude that the maximal simplices of $K$ are
\[ \tau = \{0, \dots, k-1\}, \quad \tau' = \{n-k-1, \dots, n-2\}, \quad\mbox{and}\quad \tau_{i,j} = \{i, \dots, j+k-n\} \cup \{j, \dots, i+k\}, \]
subject to the conditions
\[ 0 \le i \le j+k-n \quad\mbox{and}\quad j \le i+k \le n-2. \]
\begin{figure}[h!]
	\begin{tabular}{ccc}
	\includegraphics[scale=0.8]{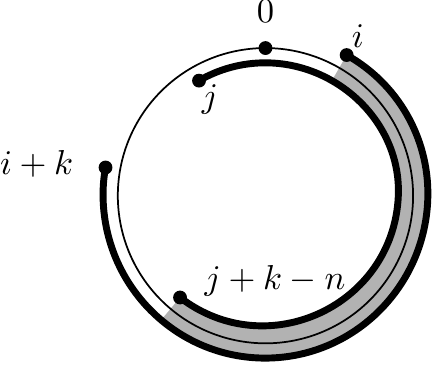} & \includegraphics[scale=0.8]{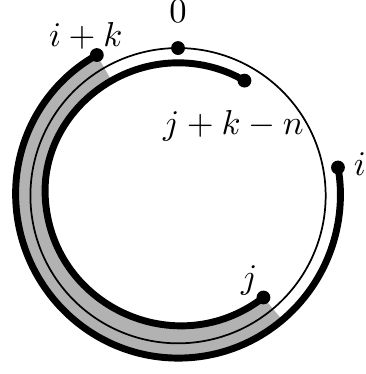} & \includegraphics[scale=0.8]{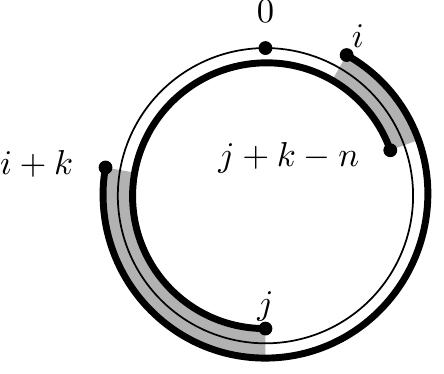}\\
	a) & b) & c)
	\end{tabular}
	\caption{\label{fig:maximalFaces}The intersections $\sigma_i\cap\sigma_j$ in $\cnk{n}{k}$.}
\end{figure}

We claim that the subcomplex $T = \tau'\cup(\bigcup_{i,j}\tau_{i,j})$ of $K$ is contractible. Let $T_l = T[\{l,\dots,n-2\}]$ for $l=0,\dots,n-k-1$. For $l \neq n-k-1$ the maximal simplices of $T_l$ containing $l$ are of the form $\tau_{l,j}$, since a maximal simplex of $T$ containing $l$ is of the form $\tau_{i,j}$ for some $i\le l\le j+k-n$ and $\tau_{i,j}\cap V(T_l)\subseteq \tau_{l,j}$. Since each $\tau_{l,j}$ contains $l+k$, vertex $l$ is dominated by $l+k$ in $T_l$, giving $T_l \simeq T_l \setminus \{l\} = T_{l+1}$. It follows that $T = T_0$ is homotopy equivalent to $T_{n-k-1}=\tau'$, which is is contractible. 

We write $K=\tau\cup T$ as the union of two contractible subcomplexes, and by Lemma~\ref{lem:union} there is a homotopy equivalence $K\simeq \Sigma\,(\tau \cap T)$. Note the vertex set of $\tau \cap T$ is $\{0,\dots,k-1\}$, and its maximal simplices are the inclusion-wise maximal elements in the family consisting of $\tau \cap \tau'$ and all $\tau\cap \tau_{i,j}$. These maximal elements are
\begin{alignat*}{2}
\tau \cap \tau' &= \{n-k-1, \dots, k-1\} \\
\tau \cap \tau_{0,j} &= \{j, \dots, k-1\} \cup \{0, \dots, j+k-n\}, \qquad && n-k \le j \le k-1\\
\tau \cap \tau_{i,i+k} &= \{i, \dots, i+2k-n\}, && 0 \le i \le n-k-2.
\end{alignat*}
These are precisely all the cyclic intervals of the form $[i,i+(2k-n)]_k$ in $\{0,\dots,k-1\}$, hence $\tau \cap T = \cnk{k}{2k-n}$. By combining the two suspension steps we obtain 
\[ \cnk{n}{k}\simeq\Sigma\,K\simeq\Sigma^2\,(\tau \cap T)=\Sigma^2\,\cnk{k}{2k-n}. \]
\end{proof}

The homotopy types of the nerve complexes $\cnk{n}{k}$ follow.

\begin{theorem}\label{thm:cnk}
Let $0 \le k \le n-2$. Then
\[ \cnk{n}{k} \simeq
\begin{cases}
\bigvee^{n-k-1}S^{2l} & \mbox{if } \frac{k}{n} = \frac{l}{l+1} \\
S^{2l+1} & \mbox{if } \frac{l}{l+1} < \frac{k}{n} < \frac{l+1}{l+2}
\end{cases}
\mbox{ for some }l \ge 0.\]
\end{theorem}

\begin{proof}
We apply Proposition~\ref{prop:cnk_susp} repeatedly to the two initial conditions $\cnk{n}{0} \cong \bigvee^{n-1}S^0$ and $\cnk{n}{k}\simeq S^1$ for $1 \le k < n/2$. For the induction step note that if $\frac{l}{l+1}< \frac{k}{n}<\frac{l+1}{l+2}$ then $\frac{l-1}{l}< \frac{2k-n}{k}<\frac{l}{l+1}$, and $\frac{k}{n}=\frac{l}{l+1}$ implies $\frac{2k-n}{k}=\frac{l-1}{l}$.
\end{proof}

\begin{figure}[h]
\begin{center}
{\small
\begin{tabular}{|>{$}c<{$}|>{$}c<{$}|>{$}c<{$}|>{$}c<{$}|>{$}c<{$}|>{$}c<{$}|>{$}c<{$}|>{$}c<{$}|>{$}c<{$}|>{$}c<{$}|>{$}c<{$}|>{$}c<{$}|>{$}c<{$}|>{$}c<{$}|>{$}c<{$}|} \hline
& k=0 & 1 & 2 & 3 & 4 & 5 & 6 & 7 & 8 & 9 & 10 & 11 & 12 \\ \hline
n=2 & \cellcolor{gray!35} S^{0} & \ast & \ast & \ast & \ast & \ast & \ast & \ast & \ast & \ast & \ast & \ast & \ast \\ \hline
3 & \cellcolor{gray!35} \vee_{2}S^{0} & S^{1} & \ast & \ast & \ast & \ast & \ast & \ast & \ast & \ast & \ast & \ast & \ast \\ \hline
4 & \cellcolor{gray!35} \vee_{3}S^{0} & S^{1} & \cellcolor{gray!35} S^{2} & \ast & \ast & \ast & \ast & \ast & \ast & \ast & \ast & \ast & \ast \\ \hline
5 & \cellcolor{gray!35} \vee_{4}S^{0} & S^{1} & S^{1} & S^{3} & \ast & \ast & \ast & \ast & \ast & \ast & \ast & \ast & \ast \\ \hline
6 & \cellcolor{gray!35} \vee_{5}S^{0} & S^{1} & S^{1} & \cellcolor{gray!35} \vee_{2}S^{2} & \cellcolor{gray!35} S^{4} & \ast & \ast & \ast & \ast & \ast & \ast & \ast & \ast \\ \hline
7 & \cellcolor{gray!35} \vee_{6}S^{0} & S^{1} & S^{1} & S^{1} & S^{3} & S^{5} & \ast & \ast & \ast & \ast & \ast & \ast & \ast \\ \hline
8 & \cellcolor{gray!35} \vee_{7}S^{0} & S^{1} & S^{1} & S^{1} & \cellcolor{gray!35} \vee_{3}S^{2} & S^{3} & \cellcolor{gray!35} S^{6} & \ast & \ast & \ast & \ast & \ast & \ast \\ \hline
9 & \cellcolor{gray!35} \vee_{8}S^{0} & S^{1} & S^{1} & S^{1} & S^{1} & S^{3} & \cellcolor{gray!35} \vee_{2}S^{4} & S^{7} & \ast & \ast & \ast & \ast & \ast \\ \hline
10 & \cellcolor{gray!35} \vee_{9}S^{0} & S^{1} & S^{1} & S^{1} & S^{1} & \cellcolor{gray!35} \vee_{4}S^{2} & S^{3} & S^{5} & \cellcolor{gray!35} S^{8} & \ast & \ast & \ast & \ast \\ \hline
11 & \cellcolor{gray!35} \vee_{10}S^{0} & S^{1} & S^{1} & S^{1} & S^{1} & S^{1} & S^{3} & S^{3} & S^{5} & S^{9} & \ast & \ast & \ast \\ \hline
12 & \cellcolor{gray!35} \vee_{11}S^{0} & S^{1} & S^{1} & S^{1} & S^{1} & S^{1} & \cellcolor{gray!35} \vee_{5}S^{2} & S^{3} & \cellcolor{gray!35} \vee_{3}S^{4} & \cellcolor{gray!35} \vee_{2}S^{6} & \cellcolor{gray!35} S^{10} & \ast & \ast \\ \hline
13 & \cellcolor{gray!35} \vee_{12}S^{0} & S^{1} & S^{1} & S^{1} & S^{1} & S^{1} & S^{1} & S^{3} & S^{3} & S^{5} & S^{7} & S^{11} & \ast \\ \hline
14 & \cellcolor{gray!35} \vee_{13}S^{0} & S^{1} & S^{1} & S^{1} & S^{1} & S^{1} & S^{1} & \cellcolor{gray!35} \vee_{6}S^{2} & S^{3} & S^{3} & S^{5} & S^{7} & \cellcolor{gray!35} S^{12} \\ \hline
15 & \cellcolor{gray!35} \vee_{14}S^{0} & S^{1} & S^{1} & S^{1} & S^{1} & S^{1} & S^{1} & S^{1} & S^{3} & S^{3} & \cellcolor{gray!35} \vee_{4}S^{4} & S^{5} & \cellcolor{gray!35} \vee_{2}S^{8} \\ \hline
16 & \cellcolor{gray!35} \vee_{15}S^{0} & S^{1} & S^{1} & S^{1} & S^{1} & S^{1} & S^{1} & S^{1} & \cellcolor{gray!35} \vee_{7}S^{2} & S^{3} & S^{3} & S^{5} & \cellcolor{gray!35} \vee_{3}S^{6} \\ \hline
17 & \cellcolor{gray!35} \vee_{16}S^{0} & S^{1} & S^{1} & S^{1} & S^{1} & S^{1} & S^{1} & S^{1} & S^{1} & S^{3} & S^{3} & S^{3} & S^{5} \\ \hline
18 & \cellcolor{gray!35} \vee_{17}S^{0} & S^{1} & S^{1} & S^{1} & S^{1} & S^{1} & S^{1} & S^{1} & S^{1} & \cellcolor{gray!35} \vee_{8}S^{2} & S^{3} & S^{3} & \cellcolor{gray!35} \vee_{5}S^{4} \\ \hline
\end{tabular}
}
\end{center}
\caption{The homotopy types of nerve complexes $\cnk{n}{k}$.}
\end{figure}

\begin{remark}
Let the boundary of a pure $d$-dimensional simplicial complex be the subcomplex induced by all $(d-1)$-faces that are contained in exactly one maximal simplex. For $k\geq 2$ and $n\geq 2k+1$ the boundary of $\cnk{n}{k}$ (denoted $M_{k-2}^{k-1}(n)$ by K{\"{u}}hnel and Lassmann \cite{KuhnelLassmann1996} and $X_n^{k-1}(\psi_0)$ by Bagchi and Datta \cite{BagchiDatta2008}) is a triangulation of the sphere product $S^{k-2} \times S^1$ if $k$ is odd or if $n$ is even and a triangulation of the analogous ``twisted sphere product" if $k$ is even and $n$ is odd. This fact is a combination of \cite[Section 5, part (c) of the Theorem]{KuhnelLassmann1996} and \cite[Lemma 3.3]{BagchiDatta2008}. It can also be obtained by extending the argument of K{\"{u}}hnel \cite{Kuhnel1986} for the case $n=2k+1$.
\end{remark}

\section{Application to the Lov\'{a}sz bound}\label{sec:Lovasz}

Let $G$ be a  simple graph with vertex set $V(G)$. The \emph{chromatic number} $\chi(G)$ is the smallest number of colors required to color the vertices of $G$ so that no two adjacent vertices have the same color. The \emph{neighborhood complex} $N(G)$ is the simplicial complex whose vertex set is $V(G)$ and whose simplices are those subsets of $V(G)$ which have a common neighbor. For $Y$ a topological space, we let $\conn(Y) \geq -1$ be the minimum $k$ such that $Y$ is $k$-connected, i.e.\ the first $k$ homotopy groups of $Y$ are trivial. The following lower bound, due to Lov\'{a}sz, is now a classical result in topological combinatorics.

\begin{theorem}[Lov\'{a}sz \cite{Lovasz1978}]\label{thm:lovasz}
For $G$ a graph we have $\chi(G) \ge \conn(N(G))+3$.
\end{theorem}
We will use our understanding of the nerve complexes $\cnk{n}{k}$ to exhibit a certain natural family of graphs which attain equality in Theorem~\ref{thm:lovasz}.

For $n\ge 2d$, the \emph{circular complete graph} $K_{n/d}$ has vertex set $\{0, \dots, n-1\}$, and two vertices $i$ and $j$ are adjacent if and only if $d \le |i-j| \le n-d$. The \emph{circular chromatic number} $\chi_c(G)$ is defined as
\[ \chi_c(G) = \infimum \bigl\{n/d~\big|~\mbox{there exists a graph homomorphism }G \to K_{n/d}\bigr\}. \]
We refer to the book by Hell \& Ne\v{s}et\v{r}il \cite[Section 6.1]{HellNesetril2004} for a comprehensive theory of this invariant. One can show that $\chi_c(K_{n/d}) = n/d$ and $\chi(G)=\lceil \chi_c(G) \rceil$ for any graph $G$, see \cite[Theorem 6.3, Corollary 6.11]{HellNesetril2004}, so in particular  $\chi(K_{n/d})=\lceil n/d\rceil$.
As a corollary of Theorem~\ref{thm:cnk}, we show that $\chi(K_{n/d})$ either achieves the topological lower bound of Theorem~\ref{thm:lovasz} or is off by one.

\begin{corollary}\label{cor:lovasz}
Let $n \ge 2d$. Then
\[ \chi(K_{n/d}) = \begin{cases}
\conn(N(K_{n/d}))+3 & \mathrm{if}\ 0 \le (\frac{n}{2d}\md 1) \le \frac{1}{2} \\
\conn(N(K_{n/d}))+4 & \mathrm{if}\ \frac{1}{2} < (\frac{n}{2d}\md 1) < 1.
\end{cases} \]
\end{corollary}

\begin{proof}
We have an isomorphism of simplicial complexes $N(K_{n/d}) \cong \cnk{n}{n-2d}$. There are two cases. If $n = 2dq$ then $\frac{n-2d}{n}=\frac{q-1}{q}$ and $\cnk{n}{n-2d} \simeq \bigvee^{2d-1}S^{2(q-1)}$ by Theorem~\ref{thm:cnk}. We have
\[ \chi(K_{n/d}) = \lceil n/d \rceil = 2q = \conn(N(K_{n/d}))+3. \]
If $n = 2dq+r$ with $0 < r < 2d$ then $\frac{q-1}{q}<\frac{n-2d}{n}<\frac{q}{q+1}$ and $\cnk{n}{n-2d} \simeq S^{2q-1}$ by Theorem~\ref{thm:cnk}. We have
\[ \chi(K_{n/d}) = \lceil n/d \rceil = 2q+\lceil r/d \rceil = \begin{cases}
2q+1 = \conn(N(K_{n/d}))+3 & \mbox{if }0 < r \le d \\
2q+2 = \conn(N(K_{n/d}))+4 & \mbox{if }d < r < 2d.
\end{cases} \]
\end{proof}

\section{The odd-dimensional spheres: cyclic polytopes and trigonometric polynomials}
\label{sect:cyclic-polytopes}

Our computation of the homotopy types of the nerve complexes $\cnk{n}{k}$ is based on Proposition~\ref{prop:cnk_susp}, which does not give much insight into the geometry of these complexes and does not lend itself to the study of the natural action of the dihedral group. Given that the resulting homotopy types $S^{2l+1}$ and $\bigvee^{\cdots} S^{2l}$ are quite simple it seems natural to ask if they are generated by explicit maps from spheres, or even by embedded spheres. In this section we answer this question in the positive for the odd-dimensional case, that is when $\cnk{n}{k}\simeq S^{2l+1}$.

To this end we relate the nerve complexes $\cnk{n}{k}$ to the cyclic polytopes. Consider first the case $1 \le k < \frac{n}{2}$ and project $\cnk{n}{k}$ to $\R^2$ by mapping vertices in cyclic order to those of a regular $n$-gon on the unit circle, and by extending linearly to simplices. Since the maximal simplices of $\cnk{n}{k}$ consist of vertices contained in less than half the circle, the image of the projection does not contain the origin in $\R^2$. Thus the inclusion of the bounding $n$-gon into $\R^2 \setminus \{0\}$ is a homotopy equivalence that factors through $\cnk{n}{k}$. It follows that the circle $[0,1], [1,2], \dots, [n-1,0]$ is homotopically non-trivial in $\cnk{n}{k}$. Since for $1 \le k < \frac{n}{2}$ we know $\cnk{n}{k}\simeq S^1$, we can conclude that $[0,1], [1,2], \dots, [n-1,0]$ is a homotopy equivalent subcomplex.

In this section we generalize this reasoning to other homotopy types: if $\cnk{n}{k} \simeq S^{2l+1}$ then we project to $\R^{2l+2}$. The unit circle in $\R^2$ will be generalized by the trigonometric moment curve, and the regular $n$-gons will be generalized by cyclic polytopes.

We begin by introducing cyclic polytopes; we refer the reader to Ziegler \cite{Ziegler} for the basics of polytope theory. 

\begin{definition}
The \emph{moment curve} $\gamma \colon \R\to\R^d$ is given by $\gamma(t) = (t, t^2, \dots, t^d)$. For $n>d$ the \emph{cyclic polytope} $C_d(n)$ is defined as the convex hull $C_d(n) = \conv\{\gamma(0), \gamma(1), \dots, \gamma(n-1)\}$.
\end{definition}
Gale gave a combinatorial description of the maximal simplices of cyclic polytopes.

\begin{theorem}[Gale's evenness condition, Gale \cite{Gale1963}]
Let $n > d \ge 2$. The cyclic polytope $C_d(n)$ is simplicial, and a subset $\sigma \subseteq \{0, \dots, n-1\}$ of size $d$ is a maximal simplex of $\partial C_d(n)$ if and only if for any $x, y \in \{0, \dots, n-1\}\setminus \sigma$ with $x<y$, the cardinality of $[x,y] \cap \sigma$ is even.
\end{theorem}

Note that if $d$ is even then Gale's condition can be reformulated as follows: $\sigma$ is a maximal simplex of $C_d(n)$ if and only if $\sigma$ is a disjoint union of $d/2$ sets of the form $[i,i+1]_n$. Gale furthermore remarked that in this case the cyclic polytope is combinatorially equivalent to the convex hull of points on the trigonometric moment curve.

\begin{definition}
The \emph{trigonometric moment curve} $\tgamma_{2d} \colon \R\to\R^{2d}$ is given by
\[ \tgamma_{2d}(t) = (\cos(2\pi\cdot t), \sin(2\pi\cdot t), \cos(2\pi\cdot 2t), \sin(2\pi\cdot 2t), \dots, \cos(2\pi\cdot dt), \sin(2\pi\cdot dt)). \]
\end{definition}

\begin{proposition}[Gale \cite{Gale1963}]
For $n \ge 2d+1$, the cyclic polytope $C_{2d}(n)$ is combinatorially equivalent to $\conv\{\tgamma_{2d}(0), \tgamma_{2d}(\frac{1}{n}), \dots, \tgamma_{2d}(\frac{n-1}{n})\}$.
\end{proposition}
We will restrict attention to this representation of $C_{2d}(n)$ along the trigonometric moment curve.

To relate the cyclic polytopes to our $\cnk{n}{k}$, we need to understand for which intervals $I \subset \R$ the convex hull $\conv\tgamma_{2d}(I)$ avoids the origin. This is equivalent to the existence of a separating hyperplane, i.e.\ a linear subspace $H \subseteq \R^{2d}$ of codimension one that does not intersect $\tgamma_{2d}(I)$. Such a hyperplane $H$ is determined by a normal vector $z \in \R^{2d}$ with $\langle z, \tgamma_{2d}(t) \rangle > 0$ for all $t \in I$, where $\langle \cdot, \cdot \rangle$ denotes the Euclidean inner product. 
\begin{definition}
The \emph{homogeneous trigonometric polynomial} $p_z$ of degree $d$ with coefficient vector $z \in \R^{2d}\setminus \{\vec{0}\}$ is given by
\[ p_z(t) = \langle z, \tgamma_{2d}(t) \rangle = \sum_{j=1}^d z_{2j-1} \cos(2\pi jt) + z_{2j} \sin(2\pi jt).\]
\end{definition}

\begin{theorem}[Gilbert and Smyth {\cite[Corollary 1]{GilbertSmyth2000}}]\label{thm:gaps}
For any $\theta < \frac{d}{d+1}$ there is a homogeneous trigonometric polynomial of degree $d$ that is positive on $[0, \theta]$. Moreover, no homogeneous trigonometric polynomial of degree $d$ is positive on $[0, \frac{d}{d+1}]$.
\end{theorem}

\begin{corollary}\label{cor:separating-hyperplane}
Let $I = [x,y] \subseteq \R$ be any closed interval of length less than $\frac{d}{d+1}$. Then there is a hyperplane $H \subseteq \R^{2d}$ such that $\tgamma_{2d}(I)$ is strictly on one side of $H$.
\end{corollary}

\begin{proof}
Let $\theta = y-x < \frac{d}{d+1}$ and let $p_z$ be a homogeneous trigonometric polynomial of degree $d$ that is positive on $[0, \theta]$. The function $t \mapsto p_z(t-x)$ is positive on $[x,y]$ and can be expressed as a homogeneous trigonometric polynomial of degree $d$ since the space of these functions is shift invariant. Let $\overline{z} \in \R^{2d} \setminus \{0\}$ with $p_{\overline{z}}(t) = p_z(t-x)$, and now define $H = \overline{z}^\perp$.
\end{proof}

The coefficient vector $\overline{z}$ in the proof above can be explicitly computed from $z$ by applying a simple rotation $R \colon \R^{2d} \to \R^{2d}$ to $z$ that maps $\tgamma_{2d}([0, \theta])$ to $\tgamma_{2d}([x,y])$.

The following lemma and theorem relate the cyclic polytopes to the nerve complexes $\cnk{n}{k}$.

\begin{lemma}\label{lem:polytope_inclusion}
Suppose $\frac{l}{l+1} < \frac{k}{n}$ and $n \ge 2l+3$. Then the inclusion $\partial C_{2l+2}(n) \subseteq \cnk{n}{k}$ holds for the natural ordering of vertices.
\end{lemma}

\begin{proof}
Let $\sigma$ be any maximal simplex of $C_{2l+2}(n)$, and write it as a disjoint union of $l+1$ disjoint blocks $[i,i+1]_n$ of size $2$. Then there must be two consecutive (in the cyclic sense) blocks such that the number $g$ of elements in the gap between them satisfies
$$g \geq \frac{n-2(l+1)}{l+1}=\frac{n}{l+1}-2>n-k-2,$$ where the second inequality is equivalent to $\frac{l}{l+1}<\frac{k}{n}$. Then $g\ge n-k-1$, and so there is a segment of $n-k-1$ consecutive elements in $\Z/n$ disjoint from $\sigma$, meaning $\sigma\in \cnk{n}{k}$.
\end{proof}

For example, if $l=0$ then the boundary $\partial C_2(n)\subseteq \cnk{n}{k}$ is just the $n$-cycle passing through all the vertices of $\cnk{n}{k}$ in their natural ordering. If $1\le k< n/2$ then it is not difficult to describe a deformation retraction from $\cnk{n}{k}$ to this embedded circle. Guided by this intuition we will now prove that in general the embedded $\partial C_{2l+2}(n)$ generates the homotopy type of $\cnk{n}{k}$. Note that the proof relies on the prior knowledge of $\cnk{n}{k}\simeq S^{2l+1}$, i.e.\ it cannot be used as a replacement for the proof of Theorem~\ref{thm:cnk}.

\begin{theorem}\label{thm:polytope_equivalence}
Suppose $\frac{l}{l+1} < \frac{k}{n} < \frac{l+1}{l+2}$. Then the inclusion $\partial C_{2l+2}(n) \hookrightarrow \cnk{n}{k}$ is a homotopy equivalence.
\end{theorem}

\begin{proof}
The condition $\frac{l}{l+1} < \frac{k}{n} < \frac{l+1}{l+2}$ implies $n \ge 2l+3$. Let $f \colon \cnk{n}{k} \to C_{2l+2}(n) \subset \R^{2l+2}$ be the simplex-wise affine map with $f(i) = \tgamma_{2l+2}(\frac{i}{n})$. For each $i$, the interval $[\frac{i}{n}, \frac{i+k}{n}]_{S^1}$ has length $\frac{k}{n} < \frac{l+1}{l+2}$, and so by Corollary~\ref{cor:separating-hyperplane} there is a hyperplane $H$ with $\tgamma_{2l+2}([\frac{i}{n}, \frac{i+k}{n}]_{S^1})$ strictly on one side of $H$. Hence the set $\conv\{\tgamma_{2l+2}(\frac{i}{n}), \dots, \tgamma_{2l+2}(\frac{i+k}{n})\}$ does not contain the origin, and $f$ maps to $C_{2l+2}(n) \setminus \{\vec{0}\}$.
	
The inclusion $\partial C_{2l+2}(n) \to C_{2l+2}(n) \setminus \{\vec{0}\}$ is a homotopy equivalence between $(2l+1)$-spheres that factors through $\cnk{n}{k}$: the boundary $\partial C_{2l+2}(n)$ includes into $\cnk{n}{k}$ by Lemma~\ref{lem:polytope_inclusion}, and the map $f\colon \cnk{n}{k} \to C_{2l+2}(n) \setminus \{\vec{0}\}$ is the identity on $\partial C_{2l+2}(n)$. We know $\cnk{n}{k} \simeq S^{2l+1}$ by Theorem~\ref{thm:cnk}, and it follows that $\partial C_{2l+2}(n) \hookrightarrow \cnk{n}{k}$ induces isomorphisms of homotopy groups. By Whitehead's theorem this inclusion is a homotopy equivalence.
\end{proof}

\begin{remark}
\label{remark:odd-dual}
As a consequence we have that for $\frac{l}{l+1}<\frac{k}{n}<\frac{l+1}{l+2}$ the fundamental class $[\partial C_{2l+2}(n)]$ of the embedded sphere $\partial C_{2l+2}(n)$ generates $H_{2l+1}(\cnk{n}{k})$. We will sketch another proof of this fact, which exhibits a dual generator of $H^{2l+1}(\cnk{n}{k})$.

A set $Q=\{a_1,\dots,a_{2(l+1)}\}$ will be called $(n,k)$-\emph{admissible} if it satisfies the following conditions (for convenience we also declare $a_0=0$; it is not an element of $Q$):
\begin{eqnarray*}
& & 0=a_0<a_1<\cdots<a_{2(l+1)}<n, \\
& & a_{i+1}-a_i<n-k \quad \mathrm{for}\ i=0,1,\dots,2l+1, \\
& & a_{2(i+1)}-a_{2i}\geq n-k\quad\mathrm{for}\ i=0,\dots,l, \\
& & a_{2(l+1)}-a_1\leq k.
\end{eqnarray*}
Now consider the $(2l+1)$-cochain in $\cnk{n}{k}$ given by the formula 
$\tbeta=\sum_{Q\ \mathrm{is}\ (n,k)-\mathrm{admissible}}Q^\vee.$
A rather tedious computation (omitted) shows that $\tbeta$ is in fact a cocycle. One also checks that the support of $\tbeta$ has exactly one $(2l+1)$-simplex in common with the embedded $\partial C_{2l+2}(n)$; that simplex is $\bigcup_{i=1}^{l+1}\{i(n-k)-1,i(n-k)\}$. It follows that $\langle [\tbeta],[\partial C_{2l+2}(n)]\rangle=\pm 1$ and since we know $H^{2l+1}(\cnk{n}{k})=\Z$ and $H_{2l+1}(\cnk{n}{k})=\Z$, we conclude that $[\tbeta]$ and $[\partial C_{2l+2}(n)]$ are the generators of these respective groups.
\end{remark}

We leave it as an open problem whether the embedded sphere is always a combinatorial deformation retract of $\cnk{n}{k}$.

\begin{conjecture}
\label{conj:collapses}
For $\frac{l}{l+1}<\frac{k}{n}<\frac{l+1}{l+2}$ the complex $\cnk{n}{k}$ simplicially collapses to $\partial C_{2l+2}(n)$.
\end{conjecture}

We now give an independent topological proof of the upper bound in Theorem~\ref{thm:gaps} on the distance between successive roots of a homogeneous trigonometric polynomial. This extremal problem was studied by Babenko \cite{Babenko1984}, who also showed that $\frac{d}{d+1}$ is an upper bound on the measure of the set where a homogeneous trigonometric polynomial of degree $d$ is positive. Some more general results for a further restricted set of frequencies were shown by Kozma \& Oravecz \cite{KozmaOravecz2002}. Our proof follows the argument given in the proof of Theorem~\ref{thm:polytope_equivalence} and relies on the homotopy types of the $\cnk{n}{k}$. In particular it does not resemble previously known proofs.

\begin{theorem}\label{thm:gapsAlternate}
The distance between any two successive roots in a homogeneous trigonometric polynomial of degree $d$ is at most $\frac{d}{d+1}$.
\end{theorem}

\begin{proof}
Suppose for a contradiction there is some homogeneous trigonometric polynomial of degree $d$ with two successive roots more than $\frac{d}{d+1}$ apart. After translating and changing sign, if necessary, we can find a $z \in \R^{2d} \setminus \{\vec{0}\}$ and $\theta>\frac{d}{d+1}$ such that $p_z(t)>0$ for all $t\in[0,\theta]$. Given an arbitrary $x \in \R$ there is --- by appropriately shifting $t \mapsto p_z(t-x)$ --- a coefficient vector $z(x) \in \R^{2d} \setminus \{0\}$ such that $p_{z(x)}$ is positive on $[x, x+\theta]$.

Choose integers $n \ge 2d+1$ and $k$ with $\frac{d}{d+1} < \frac{k}{n} < \theta$. Let $f\colon \cnk{n}{k} \to C_{2d}(n)\subset \R^{2d}$ be the simplex-wise affine map with $f(i) = \tgamma_{2d}(\frac{i}{n})$. For every face $\sigma$ of $\cnk{n}{k}$, some hyperplane $z(x)^\perp$ separates $f(\sigma)$ from the origin, and thus $f$ maps to $C_{2d}(n)\setminus \{\vec{0}\}$. Since $\frac{k}{n}>\frac{d}{d+1}>\frac{d-1}{d}$, Lemma~\ref{lem:polytope_inclusion} ensures that $\partial C_{2d}(n)\subseteq \cnk{n}{k}$ with the natural ordering of the vertices. It follows that the inclusion $\iota \colon \partial C_{2d}(n) \to C_{2d}(n) \setminus \{\vec{0}\}$ factors through $\cnk{n}{k}$. This is a contradiction, since $\iota$ is a homotopy equivalence between spaces homotopy equivalent to $S^{2d-1}$, but $H_{2d-1}(\cnk{n}{k})$ is trivial for $\frac{d}{d+1} < \frac{k}{n}$.
\end{proof}

\section{The even-dimensional spheres: minimal generators}

In this section we consider the case when  $\frac{k}{n} = \frac{l}{l+1}$ for some $l\geq 0$, when $\cnk{n}{k}$ is homotopy equivalent to a wedge of $2l$-spheres by Theorem~\ref{thm:cnk}. It turns out that a basis of the free abelian group $H_{2l}(\cnk{n}{k})$ can be specified using embedded spheres $\partial \Delta^{2l+1}\subset \cnk{n}{k}$. It is well-known that this is the smallest possible support a $2l$-dimensional homology class in a simplicial complex can have.

Consider the oriented $(2l+1)$-simplex
\[ \Delta = [0, 1, n-k, n-k+1, 2(n-k), 2(n-k)+1, \dots, l(n-k), l(n-k)+1]. \]
Note that $\frac{k}{n} = \frac{l}{l+1}$ implies $l(n-k)=k$ and $(l+1)(n-k)=n$. The simplex $\Delta$ is a minimal non-face of $\cnk{n}{k}$, and hence $\partial\Delta\subset \cnk{n}{k}$ is an embedded $2l$-sphere whose fundamental cycle
\[ \partial \Delta = \sum_{i=0}^l\big( \Delta \setminus \{i(n-k)\} - \Delta \setminus \{i(n-k)+1\}\big) \] is a $2l$-cycle in $\cnk{n}{k}$. Let $\alpha = [\partial \Delta] \in H_{2l}(\cnk{n}{k})$ be its homology class. We will show that $\alpha\neq 0$ by constructing a cohomology class $\beta \in H^{2l}(\cnk{n}{k})$ which pairs nontrivially with $\alpha$.

For convenience we define $I_i = \{v~|~i(n-k)<v<(i+1)(n-k)\}$ for $0 \le i \le l$. Let $\nB$ be the set of all oriented $2l$-simplices $B \in \cnk{n}{k}$ such that
\[ B = [0, v_0, n-k, v_1, 2(n-k), v_2, \dots, (l-1)(n-k), v_{l-1}, l(n-k)], \]
with $v_i \in I_i$ for all $0\leq i<l$ (we have $\nB=\{[0]\}$ when $l=0$). Consider the cochain
\[ \tbeta=\sum_{B\in\nB}B^\vee. \]
We claim that $\tbeta$ is a cocycle. Suppose for a contradiction that some $(2l+1)$-simplex $\sigma \in \cnk{n}{k}$ satisfies $(\delta\tbeta)(\sigma) = \tbeta(\partial\sigma) \neq 0$. Necessarily $i(n-k)\in \sigma$ for $0 \le i \le l$ and $|\sigma \cap I_i| \ge 1$ for $0 \le i < l$, for otherwise every face in $\partial \sigma$ is different from every simplex in the support of $\tbeta$. The condition $\sigma \in \cnk{n}{k}$ now implies  $\sigma \cap I_l = \emptyset$. Since $|\sigma| = 2l+2$, there exists an index $0 \le i' < l$ with $|\sigma \cap I_{i'}| = 2$ and $|\sigma \cap I_i| = 1$ for all $0 \le i < l$ with $i \neq i'$. Let $\sigma \cap I_{i'} = \{u,v\}$. In the sum $\tbeta(\partial \sigma)=\sum_{B \in \nB, \tau \in \partial \sigma}B^\vee(\tau)$ there are exactly two nonzero terms, namely when $B = \tau = \sigma \setminus \{u\}$ and $B = \tau = \sigma \setminus \{v\}$, and these terms appear with opposite signs since $\sigma \setminus \{u\}$ and $\sigma \setminus \{v\}$ are consecutive faces of $\sigma$.
So in fact $\tbeta(\partial \sigma) = 0$, and $\tbeta$ is a cocycle.

We define the cohomology class $\beta=[\tbeta] \in H^{2l}(\cnk{n}{k})$. Since the family $\nB$ contains exactly one simplex of the form $\Delta \setminus \{v\}$, namely for $v = l(n-k)+1$, the evaluation of $\tbeta$ on $\partial \Delta$ is $-1$. It follows that $\langle \beta, \alpha \rangle=-1$, proving $\alpha \neq 0$ and $\beta \neq 0$.

All generators of $H_{2l}(\cnk{n}{k})$ can be obtained by rotations of $\alpha$. For this, let $g$ be the generator of $\Z/n$ acting on $\cnk{n}{k}$ via the vertex map $i \mapsto (i+1)\md n$. Define $\alpha_i = [g^i(\partial\Delta)]$ and $\beta_i = [\tbeta g^{-i}]$, and note that $g^{n-k}\Delta=\Delta$ holds at the level of chains, hence $\alpha_i = \alpha_{i\md (n-k)}$ in homology. We can now prove the following.

\begin{proposition}
\label{prop:even-generators}
Suppose $\frac{k}{n}=\frac{l}{l+1}$ for some $l\geq 0$. Then the homology classes $\alpha_0,\dots,\alpha_{n-k-2}$ defined above form a basis of the group $\redhom_{2l}(\cnk{n}{k})=\Z^{n-k-1}$. Moreover
$\sum_{i=0}^{n-k-1}\alpha_{i}=0.$ 
\end{proposition}
\begin{proof}
First of all, for $0\leq i\leq n-k-1$ we have
$$\langle \beta_0, \alpha_i \rangle = \begin{cases}
-1 & \mbox{for }i=0\\
1 & \mbox{for }i=n-k-1 \\
0 & \mbox{otherwise.} \end{cases}
$$
We checked $\langle\beta_0,\alpha_0\rangle=\langle\beta,\alpha\rangle=-1$ previously. If $i\neq 0,n-k-1$ then $g^i\Delta$ avoids the vertex $0$, hence $g^i(\partial\Delta)$ pairs trivially with all simplices in $\nB$. If $i=n-k-1$ then the oriented simplex $g^{n-k-1}(\Delta\setminus\{l(n-k)\})$, which appears in $g^{n-k-1}(\partial\Delta)$ with positive sign, matches with one oriented simplex in $\nB$. 

In general we have $\langle\beta_i,\alpha_j\rangle=\langle\beta_0,\alpha_{i-j}\rangle$. Let $\gamma_i = -\sum_{s=0}^i \beta_s$ for $i = 0, \dots, n-k-2$. Then the evaluation matrix $\langle \gamma_i, \alpha_j \rangle_{i,j = 0, \dots, n-k-2}$ is the identity matrix, hence invertible over $\Z$. Since the rank of $\redhom_{2l}(\cnk{n}{k})$ is $n-k-1$, this means that the $\alpha_i$ for $0, \dots, n-k-2$ form a generating set for homology, and the $\gamma_i$ form a dual generating set for cohomology. It follows that
\[ \alpha_{n-k-1} = \sum_{i=0}^{n-k-2} \langle \gamma_i, \alpha_{n-k-1} \rangle \alpha_i
= -\sum_{i=0}^{n-k-2} \sum_{s=0}^{i} \langle \beta_s, \alpha_{n-k-1} \rangle \alpha_i
= -\sum_{i=0}^{n-k-2} \langle\beta_0,\alpha_{n-k-1} \rangle \alpha_i
= -\sum_{i=0}^{n-k-2} \alpha_i. \]

\end{proof}

\begin{remark}
Also here we can make a relation to cyclic polytopes, albeit it is not as direct as in the case of odd-dimensional spheres from the previous section. The proof of Lemma~\ref{lem:polytope_inclusion} carries over to show that $\cnk{n}{k}$ contains all maximal faces of $\partial C_{2l+2}(n)$ \emph{except} for $\Delta, g\Delta,\dots,g^{n-k-1}\Delta$. Thus $\cnk{n}{k} \cap \partial C_{2l+2}(n)$ can be obtained from $\partial C_{2l+2}(n) \cong S^{2l+1}$ by deleting $n-k$ maximal simplices, giving $\cnk{n}{k} \cap \partial C_{2l+2}(n) \simeq \bigvee^{n-k-1} S^{2l}$.
	
Part of the Mayer--Vietoris sequence for $\cnk{n}{k} \cup \partial C_{2l+2}(n)$ is
\[ \redhom_{2l}\bigl(\cnk{n}{k} \cap \partial C_{2l+2}(n)\bigr) \to \redhom_{2l}\bigl(\cnk{n}{k}\bigr) \oplus \redhom_{2l}\bigl(\partial C_{2l+2}(n)\bigr) \to \redhom_{2l}\bigl(\cnk{n}{k} \cup \partial C_{2l+2}(n)\bigr). \]
The last arrow is induced by the inclusion $\iota\colon \cnk{n}{k} \to \cnk{n}{k} \cup \partial C_{2l+2}(n)$. We proved that $\redhom_{2l}(\cnk{n}{k})$ is generated by the classes of $\partial\Delta, g(\partial\Delta),\dots,g^{n-k-2}(\partial\Delta)$ and these cycles are boundaries in $\partial C_{2l+2}(n)$, thus $\iota_{\ast} = 0$. Hence the first arrow is a surjection $\Z^{n-k-1} \to \Z^{n-k-1}$, which implies $\cnk{n}{k} \cap \partial C_{2l+2}(n) \hookrightarrow \cnk{n}{k}$ is a homotopy equivalence (again, we need to know the homotopy types of $\cnk{n}{k}$ in advance).
\end{remark}

\section{Induced representation in homology}
\label{sect:action}

We are now in position to describe the equivariant structure of $\cnk{n}{k}$. If $k\in\{0,n-2,n-1\}$ then $\cnk{n}{k}$ carries the permutation action of the full symmetric group $\Sigma_n$. In the remaining cases all the automorphisms of $\cnk{n}{k}$ are generated by the symmetries of the regular $n$-gon:

\begin{lemma}
\label{lem:aut}
If $1\leq k\leq n-3$ then the automorphism group of $\cnk{n}{k}$ is isomorphic to the dihedral group $D_{2n}$.
\end{lemma}
\begin{proof}
Let $h \colon \{0,\dots,n-1\}\to\{0,\dots,n-1\}$ be a bijection which induces an automorphism of $\cnk{n}{k}$. Then there exists some $h'$ defined by the condition that $h([i,\dots,i+k]_n)=[h'(i),\dots,h'(i)+k]_n$ for all $i$. For each $i$ the symmetric difference of $h([i,\dots,i+k]_n)$ and $h([i+1,\dots,i+1+k]_n)$ consists of two elements, namely $h(i)$ and $h(i+k+1)$. If $1\le k\le n-3$ this is possible if and only if $h'(i+1)=h'(i)\pm 1\md n$. It follows that $h'$ is an automorphism of the cycle graph $C_n$. Clearly any automorphism of $C_n$ extends to an automorphism of $\cnk{n}{k}$.
\end{proof}

The action of $D_{2n}$ on $\cnk{n}{k}$ induces an action on homology, which we now describe. Let 
\[ D_{2n}=\langle g,\varepsilon~|~g^n=1,\ \varepsilon^2 = 1,\ \varepsilon g\varepsilon=g^{-1}\rangle \]
with the action on $\{0,\dots,n-1\}$ given by $g(i)=(i+1)\md n$ and $\varepsilon(i)=-i\md n$.

\begin{proposition}
\label{prop:action-even}
Suppose that $\frac{k}{n}=\frac{l}{l+1}$ for some $l\geq 0$. Then the action of $D_{2n}$ on the elements $\alpha_i$ in $H_{2l}(\cnk{n}{k})$ is given by
$$g\alpha_i=\alpha_{i+1},\qquad \varepsilon\alpha_i=(-1)^{l+1}\alpha_{-i-1}.$$
\end{proposition}
\begin{proof}
The first equality holds by definition. For the second one, we first verify the following equation on the level of chains (note that $(l+1)(n-k)=n$):
\begin{align*}
\varepsilon\Delta&=\varepsilon[0,1,n-k,n-k+1,\dots,l(n-k),l(n-k)+1]\\
&=[0,-1,l(n-k),l(n-k)-1,\dots,n-k,(n-k)-1]\\
&=(-1)^{l+1}g^{-1}\Delta.
\end{align*}
The sign is introduced by changing the order in each pair of elements of the form $a(n-k),a(n-k)-1$, followed by reordering the pairs. Reordering pairs does not change sign.

It immediately follows that
$$\varepsilon\alpha_0=(-1)^{l+1}\alpha_{-1}$$
and then we obtain
$$\varepsilon\alpha_i=\varepsilon g^i\alpha_0=g^{-i}\varepsilon\alpha_0=(-1)^{l+1}g^{-i}\alpha_{-1}=(-1)^{l+1}\alpha_{-i-1}.$$
\end{proof}

\begin{proposition}
\label{prop:action-odd}
If $\frac{l}{l+1}<\frac{k}{n}<\frac{l+1}{l+2}$ then the action of $D_{2n}$ on $H_{2l+1}(\cnk{n}{k})=\Z$ is given by
$$g=\mathrm{id},\qquad \varepsilon=(-1)^{l+1}\cdot\mathrm{id}.$$
\end{proposition}
\begin{proof}
We will use the notation and results of Section~\ref{sect:cyclic-polytopes}.

The action of $D_{2n}$ on $\cnk{n}{k}$ preserves the homotopy equivalent subsphere $\partial C_{2l+2}(n)$, so it suffices to compute the degree of the maps $g$ and $\varepsilon$ acting on $\partial C_{2l+2}(n)$. Recall that the cyclic polytope $C_{2l+2}(n)=\conv\{\tgamma_{2l+2}(0), \tgamma_{2l+2}(\frac{1}{n}), \dots, \tgamma_{2l+2}(\frac{n-1}{n})\}\subset \R^{2l+2}$. 

Let $\omega_{j,n} \colon \R^2\to\R^2$ be the rotation by an angle $2\pi j/n$ around the origin and define $R \colon (\R^2)^{l+1}\to(\R^2)^{l+1}$ by $R=\omega_{1,n}\oplus\cdots\oplus\omega_{l+1,n}$. Then one has $R(\tgamma_{2l+2}(\frac{i}{n}))=\tgamma_{2l+2}(\frac{i+1}{n})$, so $R|_{\partial C_{2l+2}(n)}=g$. Since $R$ is a rotation, its degree is $1$.

Next, if $E \colon \R^{2l+2}\to\R^{2l+2}$ is given by $E(x_1,x_2,\dots,x_{2l+1},x_{2l+2})=(x_1,-x_2,\dots,x_{2l+1},-x_{2l+2})$ then we easily check $E(\tgamma_{2l+2}(\frac{i}{n}))=\tgamma_{2l+2}(\frac{-i}{n})$, which implies $E|_{\partial C_{2l+2}(n)}=\varepsilon$. Since $E$ is a composition of $l+1$ hyperplane reflections, its degree is $(-1)^{l+1}$.
\end{proof}

\section{Clique complexes of evenly-spaced arcs}\label{sec:clique_even}

In this section we relate the nerve complexes of circular arcs to the clique complexes of their 1-skeletons.

\begin{definition}
For $n \ge 1$ and $k \ge 0$, we define the \emph{clique complex} $\vnk{n}{k}$ as $\vnk{n}{k} = \cl(\cnk{n}{k}^{(1)})$, i.e.\ the maximal simplicial complex with $1$-skeleton $\cnk{n}{k}^{(1)}$.
\end{definition}

If $\nU_{n,k}$ is a collection of evenly-spaced arcs defined as in \eqref{eq:evenly-spaced} and $X_n\subset S^1$ is a set of $n$ equally-spaced points, then analogous to the sequence of isomorphisms $\cech(X_n,S^1;\frac{k}{2n})\cong\nN(\nU_{n,k})\cong\cnk{n}{k}$ in \eqref{eq:cech-iso} we have
\begin{equation}
\label{eq:clique-iso}
\vr\bigl(X_n,\tfrac{k}{n}\bigr)\cong\onN(\nU_{n,k})\cong\vnk{n}{k}.
\end{equation}
Note that $\vnk{n}{k}$ is a full simplex when $k\geq \lfloor n/2 \rfloor$.

The $1$-skeleton $\cnk{n}{k}^{(1)}=\vnk{n}{k}^{(1)}$ is the graph commonly denoted as $C_n^k$, the \emph{$k$-th distance power of the cycle $C_n$}. In this graph the neighborhood of vertex $i$ is $[i-k,i-1]_n\cup[i+1,i+k]_n$. The homotopy types of $\cl(C_n^k)=\vnk{n}{k}$ were determined by the first author.

\begin{theorem}[Adamaszek, {\cite[Corollary~6.7]{Adamaszek2013}}]\label{thm:cor6.7}
Let $ 0 \le k < n/2$. Then
\[ \vnk{n}{k} \simeq
\begin{cases}
\bigvee^{n-2k-1} S^{2l} & \mbox{if } \frac{k}{n} = \frac{l}{2l+1} \\
S^{2l+1} & \mbox{if } \frac{l}{2l+1} < \frac{k}{n} < \frac{l+1}{2l+3} 
\end{cases} 
\mbox{ for some }l \ge 0.\]
\end{theorem}

\begin{example}\label{ex:vnk93}
The clique complex $\vnk{9}{3}$ is homotopy equivalent to $\bigvee^2 S^2$. To visualize this, note that $\vnk{9}{3}$ has nine maximal 3-simplices and three maximal 2-simplices. Let $Y \simeq S^1$ be the union of the nine maximal 3-simplices. The three maximal 2-simplices are glued along their boundaries to $Y$ by maps homotopic to the identity of $S^1$, giving $\vnk{9}{3} \simeq \bigvee^2 S^2$.
\begin{figure}[h]
	\begin{center}
    	\includegraphics[width=3.5in]{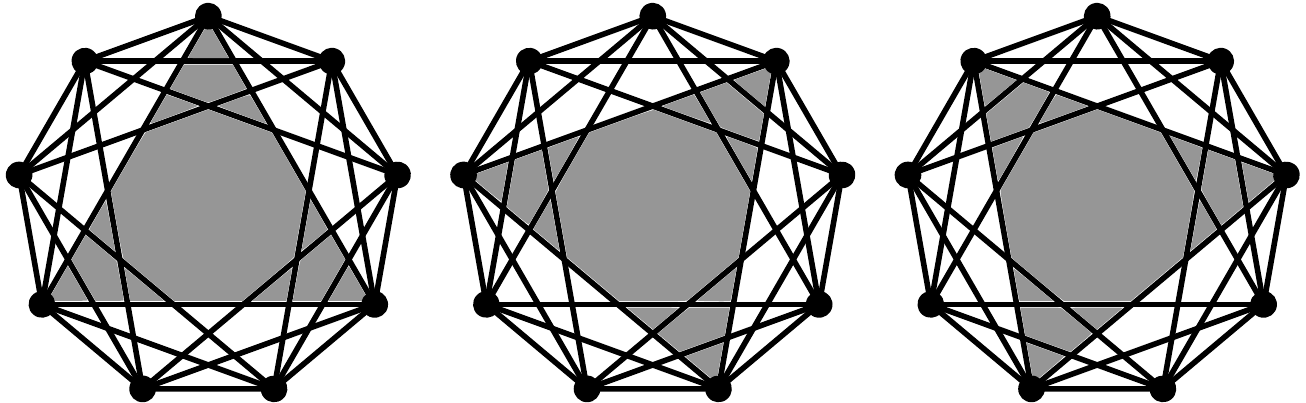}
	\end{center}
	\caption{The three maximal 2-simplices of $\vnk{9}{3}$.}
	\label{fig:vnk93}
\end{figure}
\end{example}

\begin{remark}
Note that for $0\leq k<n/2$ the clique complexes $\vnk{n}{k}$ go through their entire range of homotopy types, while the nerve complexes $\cnk{n}{k}$ remain homotopy equivalent to $S^1$. By the time when the $\cnk{n}{k}$ attain interesting homotopy types, that is when $k\geq n/2$, the clique complexes $\vnk{n}{k}$ have already become contractible.
\end{remark}

A careful comparison of Theorem~\ref{thm:cor6.7} and Theorem~\ref{thm:cnk} reveals that there is a homotopy equivalence $\vnk{n+k}{k} \simeq \cnk{n}{k}$. In the next theorem we show that this equivalence is realized by a rather surprising map. The proof of Theorem~\ref{thm:vnk_cnk} is independent of the knowledge of the homotopy types of $\cnk{n}{k}$ and $\vnk{n+k}{k}$, so combining it with Theorem~\ref{thm:cnk} gives an alternate proof of Theorem~\ref{thm:cor6.7}, which we regard as simpler and more self-contained.

\begin{theorem}\label{thm:vnk_cnk}
Let $n \ge 1$ and $k \ge 0$. The assignment $f \colon \{0,\dots,n+k-1\}\to\{0,\dots,n-1\}$ via $f(i)=i\md n$ determines a simplicial, surjective homotopy equivalence
\[ f \colon \vnk{n+k}{k} \xrightarrow{\simeq} \cnk{n}{k}. \]
\end{theorem}

\begin{proof}
If $k \ge n-1$, then both $\vnk{n+k}{k}$ and $\cnk{n}{k}$ are simplices, and hence it suffices to consider the case $k \le n-2$. We recall that if $\sigma\in\vnk{n+k}{k}$ and $t\in\sigma$ is an arbitrary vertex then $\sigma\subset[t-k,t+k]_{n+k}$.

We first verify that $f$ is a simplicial map. 
Let $\sigma$ be any simplex in $\vnk{n+k}{k}$ and set $t = \min(\sigma)$. If $t \ge k$ then $\sigma \subseteq [t, t+k]_{n+k}$, because $\sigma\cap[t-k,t-1]_{n+k}\subset\sigma\cap[0,t-1]_{n+k}=\emptyset$. After applying $f$ we have  $f(\sigma) \subseteq [t, t+k]_n$. If $t < k$, then we have
\begin{align*}
\sigma &\subseteq [t, t+k]_{n+k} \cup [t-k \md (n+k), n+k-1]_{n+k} \\
&= [t, t+k]_{n+k} \cup [t+n, n+k-1]_{n+k}.
\end{align*}
Applying map $f$ gives
\[ f(\sigma) \subseteq [t, t+k]_n \cup [t, k-1]_n \subseteq [t, t+k]_n. \]
In each case we have that $f(\sigma)$ is a face of $\cnk{n}{k}$. To verify surjectivity, note that for $0 \le i < n$ we have $f([i, i+k]_{n+k}) = [i, i+k]_n$, and so each maximal simplex of $\cnk{n}{k}$ is in the image of $f$.

It remains to show that $f$ is a homotopy equivalence. We use a simplicial variant of Quillen's Theorem A due to Barmak \cite[Theorem~4.2]{BarmakTheoremA2011}, which states that if $f$ is a simplicial map such that the preimage of every simplex is contractible, then $f$ is a (simple) homotopy equivalence. Consider an arbitrary simplex
\[ \tau=\{i_1,\dots,i_s\}\cup\{j_1,\dots,j_t\} \]
in $\cnk{n}{k}$, with
\[ 0\le i_1 < \cdots < i_s< k \le j_1 < \cdots < j_t \le n-1. \]
The preimage $f^{-1}(\tau)$ is the subcomplex of $\vnk{n+k}{k}$ induced by the vertex set
\[ V(f^{-1}(\tau)) = \{i_1, \dots, i_s\} \cup \{j_1, \dots, j_t\} \cup \{i_1+n, \dots, i_s+n\}. \]
We will show the slightly stronger statement that $f^{-1}(\tau)$ is a cone. For this it suffices to find a vertex $w\in f^{-1}(\tau)$ such that $f^{-1}(\tau)\subset [w-k,w+k]_{n+k}$, since then $w$ is adjacent, in the $1$-skeleton, to all vertices of $f^{-1}(\tau)$. 

\begin{figure}[h!]
	\begin{tabular}{cc}
	\includegraphics[scale=0.8]{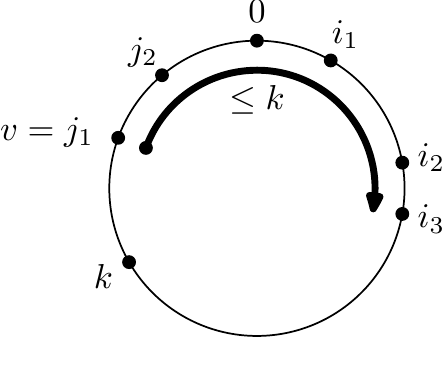} & \includegraphics[scale=0.8]{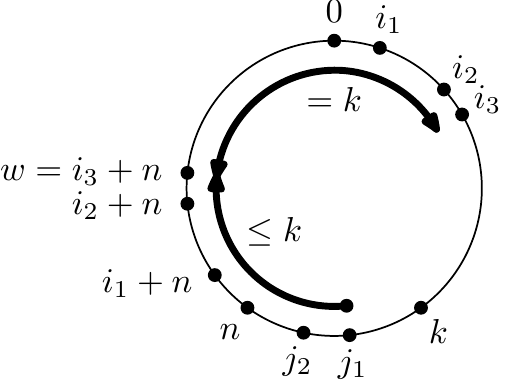}
	\end{tabular}
	\caption{\label{fig:preimage}The vertex set of the simplex $\tau$ (left, on a cycle of length $n$) and the vertex set of $f^{-1}(\tau)$ (right, on a cycle of length $n+k$). The arcs form a graphical representation of the proof in case (iii). The remaining parts are similar.}
\end{figure}

Choose any vertex $v\in\tau$ such that $\tau\subset [v,v+k]_n$. We have the following cases (see Figure~\ref{fig:preimage}).
\begin{itemize}
\item[(i)] If $v = i_q$ for some $1\le q\le s$, then take $w=i_q$. We have
\begin{align*}
\{i_q+n, \ldots, i_s+n\} \cup \{i_1, \ldots, i_q\} &\subseteq [i_q-k,i_q]_{n+k} \\
\{i_q, \ldots, i_s\} \cup \{j_1, \ldots, j_t\} \cup \{i_1+n, \ldots, i_{q-1}+n\} &\subseteq [i_q,i_q+k]_{n+k}.
\end{align*}
\item[(ii)] If $v = j_1$ and $s=0$, then take $w=j_1$. We have $\{j_1, \ldots, j_t\} \subseteq [j_1,j_1+k]_{n+k}$.
\item[(iii)] If $v = j_1$ and $s>0$, then take $w=i_s+n$. We have
\begin{align*}
\{j_1, \ldots, j_t\} \cup \{i_1+n, \ldots, i_s+n\} &\subseteq [i_s+n-k,i_s+n]_{n+k} \\
\{i_1, \ldots, i_s\} &\subseteq [i_s+n,i_s+n+k]_{n+k}.
\end{align*}
\item[(iv)] The remaining case, $v = j_q$ with $q\ge 2$, is impossible, as $j_{q-1}\not\in[j_q,j_q+k]_n$.
\end{itemize}
This completes the proof.
\end{proof}

\begin{remark}
The typical nesting between \u Cech and Vietoris--Rips complexes, for example in Carlsson \cite[Proposition~2.6]{Carlsson2009}, here takes the form
\[
\genfrac{}{}{0pt}{0}{\phantom{\subseteq}\cech\bigl(X_n,S^1;\frac{k}{2n}\bigr)\subseteq}{{\overset{\mkern4mu\verteq}{\cnk{n}{k}}}}
\genfrac{}{}{0pt}{0}{\vr\bigl(X_n,\frac{k}{n}\bigr)}{{\overset{\mkern4mu\verteq}{\vnk{n}{k}}}}
\genfrac{}{}{0pt}{0}{\subseteq\cech\bigl(X_n,S^1;\frac{k}{n}\bigr).\phantom{\subseteq}}{{\overset{\mkern4mu\verteq}{\cnk{n}{2k}}}}
\]
These inclusions relate \u Cech and Vietoris--Rips complexes on the same vertex set. The map giving the homotopy equivalence in Theorem~\ref{thm:vnk_cnk} is not an inclusion of this form: in particular it relates Vietoris--Rips complex $\vr(X_{n+k},\frac{k}{n+k})$ and ambient \u Cech complex $\cech(X_n,S^1;\frac{k}{2n})$ on vertex sets of different sizes.
\end{remark}

\section{Nerve and clique complexes of arbitrary circular arcs}\label{sec:arbitrary}

So far we analyzed the spaces $\nN(\nU)$ and $\onN(\nU)$ when $\nU=\nU_{n,k}$ is an evenly-spaced configuration of arcs in \eqref{eq:evenly-spaced}. In this section we prove that for an \emph{arbitrary} finite collection of arcs $\nU$ in $S^1$, complexes $\nN(\nU)$ and $\onN(\nU)$ are homotopy equivalent to a point, an odd-dimensional sphere, or a wedge sum of spheres of the same even dimension. This applies, in particular, to the ambient \u Cech complex $\cech(X,S^1;r)$ and Vietoris--Rips complex $\vr(X,r)$ of any finite subset $X\subset S^1$. We achieve this by showing that successively removing dominated vertices from any complex $\nN(\nU)$ (resp.\ $\onN(\nU)$) produces a complex isomorphic to some $\cnk{n}{k}$ (resp.\ $\vnk{n}{k}$), at which point the homotopy type can be read off from Theorem~\ref{thm:cnk} or \ref{thm:cor6.7}. If $\nU$ has $n$ arcs, then this reduction procedure, and therefore the computation of the homotopy type of $\nN(\nU)$ or $\onN(\nU)$, can be computed in time $O(n\log n)$, or in time $O(n)$ if the endpoints of the arcs are given in cyclic order. It follows that for $X \subset S^1$ of size $n$, the homotopy type of $\cech(X,S^1;r)$ or $\vr(X,r)$ can be computed in time $O(n\log n)$.

First we introduce additional notation. For $x_1, \dots, x_k \in S^1$, we write $(x_1 \preceq \dots \preceq x_k \preceq x_1)$ if the points $x_1, \dots, x_k$ are ordered in a clockwise fashion (allowing equality). We replace ``$x_i \preceq x_{i+1}$" with ``$x_i \prec x_{i+1}$" if furthermore $x_i \neq x_{i+1}$. Given two closed arcs $U=[a,b]_{S^1}$ and $U'=[a',b']_{S^1}$ with $a,b,a',b' \in S^1$, we write $U \preceq U'$ if $(a \preceq a' \preceq b \preceq b' \prec a)$.

\begin{lemma}\label{lem:dominatedArc}
Let $\nU = \{U_i = [a_i,b_i]_{S^1}~|~i=0, \dots, n-1\mbox{ and }a_i,b_i \in S^1\}$ be a collection of $n$ closed circular arcs. If some $i \neq j$ satisfies
\begin{enumerate}
\item[(a)] $U_i \subseteq U_j$,
\item[(b)] $U_i \preceq U_j$ and $b_k \notin [a_i, a_j)_{S^1}$ for all $k$, or
\item[(c)] $U_j \preceq U_i$ and $a_k \notin (b_j, b_i]_{S^1}$ for all $k$,
\end{enumerate}
then vertex $i$ is dominated by vertex $j$ in $\nN(\nU)$.
\end{lemma}
\begin{proof}
Suppose $\sigma \in \lk_{\nN(\nU)}(i)$ and let $U_\sigma = \bigcap_{k \in \sigma}U_k$, so that $U_\sigma \cap U_i \neq \emptyset$. We claim $U_\sigma \cap U_i \cap U_j \neq \emptyset$. Case (a) is clear. In case (b), suppose for a contradiction that $U_\sigma \cap U_i \subseteq [a_i, a_j)_{S^1}$. But then $b_k \notin [a_i, a_j)_{S^1}$ for all $k \in \sigma \cup \{i\}$ gives $a_j \in U_\sigma \cap U_i$, a contradiction. Case (c) follows by symmetry.
\end{proof}

\begin{figure}[h!]
\includegraphics[scale=0.7]{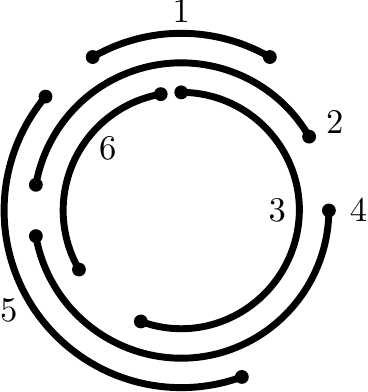}
\caption{\label{fig:arcs-example}Example: a collection of six arcs. In the nerve $1$ is dominated by $2$ as per Lemma~\ref{lem:dominatedArc}(a) and $4$ is dominated by $5$ as per Lemma~\ref{lem:dominatedArc}(b). The nerve of the subcollection $\{2,3,5,6\}$ is isomorphic to $\cnk{4}{1}$.}
\end{figure}

Recall that a simplicial complex is \emph{minimal} if it contains no dominated vertices.

\begin{proposition}\label{prop:minimalArc}
Let $\nU$ be a nonempty finite collection of arcs in $S^1$. If $\nN(\nU)$ is minimal, then there is an isomorphism $\nN(\nU) \cong \cnk{n}{k}$ for some $0 \le k < n$. 
\end{proposition}

\begin{proof}
We claim that we may restrict to the case where $\nU = \{U_i = [a_i,b_i]_{S^1}~|~i=0, \dots, n-1\mbox{ and }a_i,b_i \in S^1\}$ is a collection of $n$ closed circular arcs with all endpoints distinct, meaning $a_i \neq a_j$, $b_i \neq b_j$, and $a_i \neq b_j$ for all $i \neq j$, and also $a_i \neq b_i$ for all $i$. This is because for any collection $\nU'$ of possibly open, half-open, or closed arcs, there exists a collection $\nU$ of closed arcs with distinct endpoints such that $\nN(\nU) \cong \nN(\nU')$. 

Without loss of generality, order the arcs of $\nU$ so that $(a_0 \prec \dots \prec a_{n-1} \prec a_0)$. No arc contains another by Lemma~\ref{lem:dominatedArc}(a), and it follows that $(b_0 \prec \dots \prec b_{n-1} \prec b_0)$. We refer to the $a_i$ as \emph{opening} endpoints and to the $b_i$ as \emph{closing} endpoints. We claim that when cyclically ordered, the set of all endpoints must alternate between opening endpoints and closing endpoints. Suppose for a contradiction that there were no closing endpoint between $a_i$ and $a_{i+1\md n}$. Then by Lemma~\ref{lem:dominatedArc}(b), vertex $i+1\md n$ would dominate vertex $i$, a contradiction. Since the number of opening endpoints is equal to the number of closing endpoints, there must also be an opening endpoint between each $b_i$ and $b_{i+1\md n}$. It follows that there is some constant $0 \leq k<n$ with
\[ (a_0 \prec b_{-k\md n} \prec a_1 \prec \ldots \prec a_i \prec b_{i-k\md n} \prec a_{i+1} \prec \ldots \prec a_{n-1} \prec b_{n-1-k\md n} \prec a_0).\]
The maximal simplices of $\nN(\nU)$ are given by the nonempty intersections $U_i \cap \ldots \cap U_{i+k\md n}$ for $i = 0, \ldots, n-1$, and hence $\nN(\nU) \cong \cnk{n}{k}$.
\end{proof}

The following lemma allows us to extend the result to clique complexes.

\begin{lemma}\label{lem:dominatedClique}
If $v$ is dominated by $v'$ in $K$, then $v$ is dominated by $v'$ in $\cl(K^{(1)})$.
\end{lemma}

\begin{proof}
Suppose $\sigma \in \lk_{\cl(K^{(1)})}(v)$, hence the complete graph on vertex set $\sigma \cup \{v\}$ is in $K$. Since $\lk_K(v)$ is a cone with apex $v'$, the complete graph on vertex set $\sigma \cup \{v,v'\}$ is also in $K$, giving $\sigma \cup \{v'\} \in \lk_{\cl(K^{(1)})}(v)$.
\end{proof}

Now we can prove the main result of this section.

\begin{theorem}\label{thm:arbitraryNerve}
Let $\nU$ be a nonempty finite collection of arcs in $S^1$. Then there exist integers $n\geq 1$, $k\geq 0$ such that $\nN(\nU)\simeq \cnk{n}{k}$ and $\onN(\nU)\simeq \vnk{n}{k}$. 

In particular, $\nN(\nU)$ and $\onN(\nU)$ have the homotopy type of a point, an odd-dimensional sphere, or a wedge sum of spheres of the same even dimension.
\end{theorem}

\begin{proof}
We sequentially remove dominated vertices from $\nN(\nU)$ until we obtain a subcollection $\nU'$ such that $\nN(\nU')$ is minimal. By Proposition~\ref{prop:minimalArc}, there exists some $0 \le k < n$ with $\nN(\nU) \simeq \nN(\nU') \cong \cnk{n}{k}$. 

By Lemma~\ref{lem:dominatedClique} the sequence of dominated vertices for $\nN(\nU)$ is also a sequence of dominated vertices for $\onN(\nU)$, hence $\onN(\nU)\simeq \onN(\nU')$. This implies
\[ \onN(\nU) \simeq \onN(\nU') = \cl\bigl(\nN(\nU')^{(1)}\bigr) \cong \cl\bigl(\cnk{n}{k}^{(1)}\bigr) = \vnk{n}{k}. \]
\end{proof}

\begin{remark}
The sequence of reductions from $\onN(\nU)$ to $\vnk{n}{k}$ was also obtained by Golumbic \& Hammer \cite{GolumbicHammer1988} in the context of circular arc graphs.
\end{remark}

The following corollary is a special case of Theorem~\ref{thm:arbitraryNerve} when all arcs have the same length.

\begin{corollary}\label{cor:arbitraryVR_Cech}
If $X \subset S^1$ is nonempty and finite and $0\leq r<\frac12$, then the ambient \u Cech complex $\cech(X,S^1;r)$ and the Vietoris--Rips complex $\vr(X,r)$ have the homotopy type of a point, an odd-di\-mensional sphere, or a wedge sum of spheres of the same even dimension.
\end{corollary}

\begin{proof} Let $\nU(X,r)= \{B(x,r)~|~x\in X\}$. Then $\nN(\nU(X,r)) = \cech(X,S^1;r)$ and $\onN(\nU(X,r/2)) = \vr(X,r)$ and we apply Theorem~\ref{thm:arbitraryNerve}. 
\end{proof}

The first two authors \cite{AdamaszekAdams} use Corollary~\ref{cor:arbitraryVR_Cech} to show that the ambient \u Cech complex $\cech(S^1,S^1;r)$ and the Vietoris--Rips complex $\vr(S^1,r)$ built on the infinite vertex set $S^1$ obtain the homotopy types $S^1, S^3, S^5, S^7, \dots$ as $r$ increases.

We conclude with the observation that the removals of dominated vertices in $\nN(\nU)$ can be carried out efficiently.

\begin{theorem}
\label{thm:algorithm}
Given a collection $\nU$ of $n$ circular arcs, one can compute in time $O(n\log{n})$ a subcollection $\nU'\subset \nU$ such that $\nN(\nU')\cong \cnk{|\nU'|}{k}$ for some $k$ and the inclusion $\nN(\nU')\hookrightarrow\nN(\nU)$ is a homotopy equivalence.
\end{theorem}
\begin{proof}
Without loss of generality we may restrict to the case when $\nU = \{U_i = [a_i,b_i]_{S^1}~|~i=1, \dots, n\}$ is a collection of closed arcs with all endpoints distinct. Let $L$ be a cyclic list of the $2n$ points $a_i,b_i$ in the clockwise cyclic order. The intervals $U_i$ such that $U_i\subset U_j$ for some $j\neq i$ can now be eliminated in $O(n)$ time using a standard sweep line algorithm, maintaining at each point a list of active intervals ordered by their starting points (two sweeps around the circle are sufficient to detect all inclusions). We can therefore assume that $U_i\not\subset U_j$ for $i\neq j$.

Initialize $S$ as the set of all starting points $a_i$ such that the immediate successor of $a_i$ in $L$ is some other starting point $a_j$. As long as $S\neq\emptyset$ we repeat the following: remove a point $a_i$ from $S$, delete $a_i$ and the corresponding $b_i$ from $L$ and, if the two neighbours of $b_i$ in $L$ were $a_j\prec b_i\prec a_{j'}$ (necessarily $j,j'\neq i$), we add $a_j$ to $S$. Each starting point is added to $S$ at most once, so the procedure takes $O(n)$ steps. After termination, $L$ is an alternating list of starting and ending points  for some subcollection $\nU'\subset\nU$. Since $\nU'$ does not contain nested intervals, we conclude that $\nN(\nU')\cong \cnk{|\nU'|}{k}$ as in the proof of Proposition~\ref{prop:minimalArc}. The homotopy equivalence is a consequence of Lemma~\ref{lem:dominatedArc}.
\end{proof}

\begin{remark}
The running time of the algorithm is dominated by sorting, i.e. computing the list $L$. The remaining operations take only $O(n)$ time.
\end{remark}

\begin{corollary}
Given a set $X\subset S^1$ of cardinality $n$ and $0\leq r<\frac12$, the homotopy type of the complexes $\cech(X,S^1;r)$ and $\vr(X,r)$ can be determined in time $O(n\log{n})$, or in time $O(n)$ if $X$ is given in cyclic order.
\end{corollary}

\subsection*{Acknowledgements}

We would like to thank Anton Dochtermann for encouraging us to consider the connection to the Lov\'{a}sz bound in Section~\ref{sec:Lovasz}, and we would like to thank Arnau Padrol and Yuliy Baryshnikov for helpful conversations about cyclic polytopes.

\bibliographystyle{plain}
\bibliography{NerveComplexesOfCircularArcs}

\end{document}